\documentclass[12pt]{amsart}
\usepackage{amssymb, amsmath, amsthm, latexsym, array,euscript}
\usepackage{fullpage}
\usepackage{verbatim}
\usepackage[scr=boondoxo,scrscaled=1.05]{mathalfa}

\usepackage{xcolor}
\definecolor{darkblue}{rgb}{0,0,.5}
\definecolor{darkgreen}{rgb}{.2,0.5,.2}
\usepackage[colorlinks=true,linkcolor=darkblue,urlcolor=violet,citecolor=magenta]{hyperref}

\numberwithin{equation}{section}

\newtheorem{thm}{Theorem}[section]
\newtheorem{lm}[thm]{Lemma}%[section]
\newtheorem{cl}[thm]{Corollary}
\newtheorem{prop}[thm]{Proposition}

\theoremstyle{remark}
\newtheorem{ex}[thm]{Example}%[section]
\newtheorem{rmk}[thm]{Remark}%[section]
\theoremstyle{definition}
\newcommand {\g}{{\mathfrak g}}

\newcommand {\q}{{\mathfrak q}}

\newcommand {\eus}{\EuScript}

\newcommand {\gS}{{\eus S}}
\newcommand {\gZ}{{\eus Z}}

\newcommand {\esi}{\varepsilon}

\newcommand {\vp}{\varphi}
\newcommand {\vth}{\vartheta}

\newcommand{\ard}{\rightsquigarrow}

\newcommand{\gt}{\mathfrak}

\newcommand{\mW}{{\mathbb W}}

\newcommand{\GL}{{\rm GL}}

\newcommand{\Aut}{\mathsf{Aut}}
\newcommand{\Ann}{\mathrm{Ann}}
\newcommand{\ad}{\mathrm{ad}}

\newcommand{\id}{{\rm id}}
\newcommand{\ind}{{\rm ind\,}}

\newcommand{\codim}{\mathrm{codim\,}}
\newcommand{\rk}{\mathrm{rk\,}}

\newcommand{\Lie}{\mathsf{Lie\,}}

\newcommand{\gr}{\mathrm{gr\,}}

\newcommand {\cL}{{\mathcal L}}

\newcommand {\cz}{{\mathcal Z}}

\newcommand {\gzu}{{\gZ\!\mathscr{v}\!}}

\newcommand {\trdeg}{{\mathrm{tr.deg\,}}}

\newcommand {\mK}{{\Bbbk}}
\newcommand {\bbk}{{\Bbbk}}

\newcommand {\Z}{{\mathbb Z}}

\newcommand {\BZ}{{\mathbb Z}}

\newcommand {\beq}{\begin{equation}}
\newcommand {\eeq}{\end{equation}}

\renewcommand{\le}{\leqslant}
\renewcommand{\ge}{\geqslant}

\newcommand{\U}{{\eus U}}

\newcommand{\bff}{\boldsymbol{f}}
\newcommand{\bh}{\boldsymbol{h}}

\newcommand{\wq}{{\widehat{\q }}}
\newcommand{\wg}{{\widehat{\g }}}

\begin{document}
\hfill {\scriptsize July 10, 2025} %%\today}
\vskip1ex

\title[Current algebras and related Poisson-commutative subalgebras]{Invariants of twisted current algebras %%, finite order automorphisms, 
\\ and related Poisson-commutative subalgebras}
\author[D.\,Panyushev]{Dmitri I. Panyushev}
\address[D.P.]%
{Higher School  of Modern Mathematics, MIPT
Moscow 115184, Russia}
\email{panyush@mccme.ru}
\author[O.\,Yakimova]{Oksana S.~Yakimova}
\address[O.Y.]{Institut f\"ur Mathematik, Friedrich-Schiller-Universit\"at Jena,  07737 Jena,
Deutschland}
\email{oksana.yakimova@uni-jena.de}
\thanks{%%The ...
The second author is funded by the DFG (German Research Foundation) --- project number 404144169.}
%%, Poisson-commutative subalgebra}
\begin{abstract}
Let $\q$ be a finite-dimensional Lie algebra and $\vth$  an  automorphism of order $m$ of $\q$. 
We extend $\vth$ to an automorphism of the loop algebra $\wq=\q[t,t^{-1}]$ and consider 
the fixed-point subalgebra $\wq^\vth$. 

Using the splitting  $\wq^\vth=\q[t]^\vth\oplus (t^{-1}\q[t^{-1}])^\vth$, 
we construct  subalgebras of the symmetric algebra 
$\gZ_1\subset\gS(\q[t]^\vth)$ and 
$\gZ_2\subset\gS((t^{-1}\q[t^{-1}])^\vth)$. 
We prove that
if $\q$ is reductive, then both $\gZ_1$ and $\gZ_2$ are Poisson-commutative. 
Moreover, 
 $\gZ_1$ is always a polynomial ring with infinitely many generators, which are  explicitly described. The  algebras $\gZ_1$ and $\gZ_2$ are {\it $\vth$-twisted} Poisson versions of the 
universal Gaudin subalgebra introduced 
by Ilin and  Rybnikov in 2021
 and 
 the Feigin--Frenkel centre, respectively. Unlike the untwisted case there is a significant difference between 
$\gZ_1$ and $\gZ_2$.

For the natural Lie algebras homomorphism $\psi\!: \wq^\vth\to \q[t,t^{-1}]^\vth/(t^m-1) \cong\q$,
it is shown that $\psi(\gZ_1)$ and $\psi(\gZ_2)$
are  closely related to the Poisson-commutative subalgebras
of $\gS(\q)$ 
constructed  
by Panyushev--Yakimova in 2021. 
An explicit description of $\psi(\gZ_2)$ is obtained under some assumptions on $\vth$, 
while $\psi(\gZ_1)$ is described for 
all $\vth$ and all reductive $\q$.
\end{abstract}
\maketitle

%%%\tableofcontents
%%%%%%%%%%%   Introduction  %%%%%%
\section*{Introduction}
\label{sect:intro}

\noindent
The ground field $\mK$ is algebraically closed and $\mathsf{char}(\mK)=0$. 
%In this.
In this paper, $\q=\Lie Q$ is a finite-dimensional  Lie algebra and $\g=\Lie G$ is a reductive Lie algebra with
$Q$ and $G$ being connected affine algebraic groups. 
The dual space $\q^*$ is a Poisson variety, i.e., the 
algebra of polynomial functions on $\q^*$, $\bbk[\q^*]\cong \gS(\q)$, is a %% 
%bracket $\{\,\,,\,\}$. 
Poisson algebra. 
%The  algebra. 
%A. 
Poisson-commutative subalgebras of $\bbk[\q^*]$ are  important tools for the study of geometry of
the coadjoint action of $Q$ and representation theory of $\q$.

If $\wq$ is an infinite-dimensional Lie algebra, then $\gS(\wq)$ is also a Poisson algebra.  Each element of 
$\gS(\wq)$ is a polynomial function on $\wq^*$. 
Our main object of interest is the   loop algebra 
$\wq=\q[t,t^{-1}]=\q\otimes_{\bbk}\bbk[t,t^{-1}]$ associated with $\q$. 
We construct and study Poisson-commutative subalgebras of $\gS(\q[t,t^{-1}])$. 

Consider the vector space decomposition $\q[t,t^{-1}]=\q[t^{-1}]\oplus t\q[t]$. 
Both summands are Lie subalgebras, $\q[t^{-1}]$ is a current algebra. 
Identify $t\q[t]$ with the quotient space $\q[t,t^{-1}]/\q[t^{-1}]$, obtaining therefore  a natural action of 
$\q[t^{-1}]$ on $t\q[t]$. 
Here $\q t^k$ acts as zero on $\q t^s$, if $k+s\le 0$. 
Let 
$\gZ(\wq)=\gZ(\wq,t)=\gS(t\gt q[t])^{\gt q[t^{-1}]}$ be the subalgebra of $\q[t^{-1}]$-invariants in 
$\gS(t\gt q[t])$. 
%%
%%%
%%% 
An element  $Y\in\gS(t\gt q[t])$ belongs to $\gZ(\wq)$ if and only if
$\{\q t^{-k}, Y\}$ lies in the ideal $t\q[t] S(\wq)$ for all $k\ge 0$. 
By \cite[Prop.\,4.1]{mult},    %%% with $t$ and $t^{-1}$ 
$\gZ(\wq)$ is Poisson-commutative. 
In \cite{mult}, we have also considered  the algebra  $\gZ(\wq,[0]):=\gS(\gt q[t])^{\gt q[t^{-1}]}\subset\gS(\q[t])$.
 In order to define it, we regard $\gS(\gt q[t])$  as the 
quotient of the commutative algebra $\gS(\gt q[t,t^{-1}])$ by the ideal $(t^{-1}\gt q[t^{-1}])$. 
In this article,  we show that 
$\gZ(\wq,[0])$ is  Poisson-commutative as well, see Corollary~\ref{ohne}.  

Let $\q=\bigoplus_{i\in \BZ_m}\q_i$ be a $(\BZ/m\BZ)$-grading of $\q$ and $\zeta=\sqrt[m]1$ a fixed 
 primitive root of unity. Then the linear map $\vth\!:\q\to\q$ such that 
$\vth|_{\q_i}=\zeta^i{\cdot}\id$ is a finite  order  automorphism of $\q$. 
In \cite{fo}, we have constructed a Poisson-commutative subalgebra $\gZ(\q,\vth)\subset\gS(\q)$ associated with 
$\vth$. Here we consider analogues of $\gZ(\q,\vth)$ in the setting of loop algebras.  

For $k\in \mathbb Z$, let $\bar k\in\{0,1,\ldots,m{-}1\}$ be the 
residue of $k$ modulo $m$. Set 
$$
\wq^{\vth}:=\bigoplus_{k\in \BZ} \q_{\bar k}t^{-k}\subset\wq.
$$ 
Then $\wq^\vth$ is a $\BZ$-graded Lie algebra and it has negative and positive parts similarly to $\wq=\q[t,t^{-1}]$. 
Following the case of the trivial $\vth$, we split $\wq^\vth$ 
as $$
\wq^\vth=(t^{-1}\q[t^{-1}] \cap\wq^\vth)\oplus \q_{\bar 0} \oplus (t\q[t]\cap\wq^\vth)
$$
 and consider invariants of 
$\wq_+^\vth:=\wq^\vth\cap\q[t]$ or of $\q[t^{-1}]^\vth:=\wq^\vth\cap\q[t^{-1}]$.  
%%part 
Set 
$\gZ(\wq^\vth,[0])=\gS(\wq_+^\vth)^{\q [t^{-1}]^\vth}$ and  
$\gZ(\wq^\vth,t^{-1})=\gS(\wq^\vth/\wq_+^\vth)^{\wq_+^\vth}$.
Under mild assumptions on $\q$ and $\vth$, the algebra $\gZ(\wq^\vth,[0])$ is Poisson-commutative, see Theorem~\ref{twist-q}. Some other results on these two algebras, also for $\vth=\id$, have been proven before, see e.g. \cite{OY,fo,mult}.   
If $\q=\g$ is reductive, then we can say a lot more. 

Let $\wg^{\sf KM}$ be the (untwisted) {\it affine Kac--Moody algebra} associated with $\g$. Recall that 
$\wg^{\sf KM}=\g[t,t^{-1}]\oplus\bbk K \oplus \bbk C$, where $K$ is a central element, $[C,\xi t^k]=k\xi t^k$ for 
$\xi\in\g$, and 
$$
[\xi t^k,\eta t^s]=[\xi,\eta]t^{k+s}+(\xi,\eta) \delta_{k,-s} kK
$$
for some invariant scalar product $(\,\,,\,)$ on $\g$, see \cite{kac}. Note that the current algebra $\g[t]$ is a subalgebra 
of $\wg^{\sf KM}$. The construction of $\wg^{\sf KM}$ works for any reductive $\gt g$, but for our purposes the standard case of a simple $\gt g$ is sufficient.  
For subalgebras of $\gS(\g[t,t^{-1}])$, e.g. $\gS(\gt g[t])^{\gt g[t^{-1}]}$, we keep the above notation
with $\wq$ replaced by $\wg$.

In the reductive case, $\gZ(\wg)=\gr\!(\gt z(\wg))$, 
where $\gt z(\wg)=\gt z(\wg,t)$ is a large commutative subalgebra, the  {\it Feigin--Frenkel centre}, of the enveloping algebra 
$\U(\gt g[t])$. Historically, the FF-centre (abbreviation for Feigin--Frenkel)  was constructed as a %%c
 subalgebra  $\gt z(\wg,t^{-1})$ of $\U(t^{-1}\gt g[t^{-1}])$ \cite{ff}. 
 A description of $\gr\!(\gt z(\wg,t^{-1}))$ is obtained in \cite{ff}. 
 %% , the .
 It is convenient to %%switch 
 %%%
consider both algebras, $\gt z(\wg,t)$ and  $\gt z(\wg,t^{-1})$. They are isomorphic and 
an isomorphism is given by the change of variable $t^{-1}\mapsto t$. 
The use of $\gt z(\wg,t)$ makes 
 further 
changes of variables easier. %%% are possible, 
For instance, we may replace $t$ with 
$\esi t +1$, where $\esi\in\bbk$. Since $\esi$ is a scalar, 
$\xi \esi^a t^k \eta \esi^b t^s= \esi^{a+b}\xi  t^k \eta  t^s$ in $\U(\wg)$ and $\gS(\wg)$ for all 
$a,b,k,s\in \Z$ and $\xi,\eta\in\gt g$. Therefore the limit algebras 
$$
{\gzu}=\lim_{\esi\to 0} \gZ(\wg,\esi t+1)\subset \gS(\gt g[t])^{\gt g}
\ \ \text{ and } \ \ 
\widetilde{\gzu}=\lim_{\esi\to 0} \gt z(\wg,\esi t+1)\subset \U(\gt g[t])^{\gt g}
$$
are well-defined. 
In \cite{mult}, it is shown that ${\gzu}=\gZ(\wg,[0])$ and $\widetilde{\gzu}=\gt z(\wg,[0])$,
where  $\gt z(\wg,[0])$ is
 the {\it universal Gaudin subalgebra}
%%%%  a slight   
 constructed in  \cite[Sect.~5]{ir}. By 
\cite{ir}, $\gt z(\wg,[0])$
is the unique quantisation of a subalgebra $A_{\gt g}\subset\gS(\gt g[t])$ described in Section~4 loc.~cit. 
It is noticed in \cite[Sect.\,5]{mult} that $A_g$ and $\gZ(\wg,[0])$ have the same generators, i.e., 
they coincide.  
All algebras $\gZ(\wg)$, $\gt z(\wg)$, $\widetilde{\gzu}$, $\gzu$  are polynomial rings in infinitely many variables,
see Section~\ref{sec-ohne} for references.  

Let $p\in\bbk[t]$ be a monic polynomial of degree $n\ge 1$ such that $p(0)\ne 0$. Suppose that  the roots of 
$p$ are distinct. Then $\g\otimes_{\bbk}(\bbk[t]/(p))=:\g[t,t^{-1}]/(p)$ is isomorphic to $\g^{\oplus n}$.  An interesting observation is that 
the images of $\gt z(\wg,t^{-1})$ and $\gt z(\wg,[0])$ in $\U(\g[t,t^{-1}]/(p))$ coincide. 
This fact is mentioned in \cite[Sect.\,5]{ir} and explained in \cite[Sect.\,6.1]{mult}. 
The image in question is a Gaudin subalgebra of $\U(\g^{\oplus n})$ constructed in \cite{FFRe}. 

Let $\vth\in\Aut(\g)$ be an automorphism of finite order $m$ and $\g=\bigoplus_{i\in \BZ_m}\g_i$ the corresponding 
$(\BZ/m\BZ)$-grading with $\vth|_{\g_i}=\zeta^i{\cdot}\id$. If $(\,\,,\,)$ is $\vth$-invariant, then 
we can extend $\vth$  to an automorphism $\Theta$ of $\wg^{\sf KM}$ by setting $\Theta(C)=C$ and 
$\Theta(\xi t^k)=\zeta^k \vth(\xi) t^k$ for $\xi\in\g$, $k\in\BZ$. Then 
$(\wg^{\sf KM})^\Theta=\wg^\vth\oplus\bbk K\oplus\bbk C$ as a vector space for  $\wg^\vth=\bigoplus_{k\in \BZ} \g_{\bar k}t^{-k}$. Suppose that $\g$ is simple (non-Abelian).
If $\vth$ is an inner automorphism, then $(\wg^{\sf KM})^\Theta\cong\wg^{\sf KM}$. If $\vth$ is outer, then 
$(\wg^{\sf KM})^\Theta$ is isomorphic to a twisted affine Kac--Moody algebra. Both these statements can be found in  \cite{kac}. 
In spite of an isomorphism between $\g[t,t^{-1}]$ and the ``$\vth$-twisted" loop algebra 
$\wg^\vth$ associated to an inner $\vth$, the algebras $\gZ(\wg^\vth,t^{-1})$ and $\gZ(\wg^\vth,[0])$ differ from 
$\gZ(\wg,t^{-1})$ and $\gZ(\wg,[0])$, see Remark~\ref{Rcomp}. 

We have studied $\gZ(\wg^\vth,t^{-1})$ in \cite[Sect.\,8]{fo}. 
In the notation of that paper, our  $\gZ(\wg^\vth,t^{-1})$ is $\gZ(\widehat{\gt h}_-,\vth^{-1})$. 
If $\ind\g_{(0)}=\rk\g$, see Section~\ref{setup} for definitions, then   $\{\gZ(\wg^\vth,t^{-1}),\gZ(\wg^\vth,t^{-1})\}=0$ by  \cite[Thm.\,8.2{\sf (i)}]{fo}. Under some stronger assumptions on $\g$ and $\vth$, the algebra 
$\gZ(\wg^\vth,t^{-1})$ is a polynomial ring. However, there are automorphisms $\vth$ such that these assumptions are not satisfied and in some cases   $\gZ(\wg^\vth,t^{-1})$ is not a polynomial ring, see Remark~\ref{comp}.
%%% 
In Section~\ref{sec-twist}, we prove that $\gZ(\wg^\vth,[0])$ is always Poisson-commutative and that it has an algebraically independent infinite set of generators, see  Corollary~\ref{cl-red} and Theorem~\ref{twist-inf}. 
This shows that unlike the case of the trivial $\vth$, there is a major difference between  
$\gZ(\wg^\vth,t^{-1})$ and $\gZ(\wg^\vth,[0])$. 

In Section~\ref{Q}, we consider images of  $\gZ(\wg^\vth,t)$  and $\gZ(\wg^\vth,[0])$ in 
$\gS((\wg^\vth/(t^m-1))\cong\gS(\g)$. We show that the image of  $\gZ(\wg^\vth,[0])$ is generated by 
$\gS(\g_0)^{\g_0}$ and  a certain Poisson-commutative subalgebra 
 $\gZ_\times(\g,\vth^{-1})$ constructed in \cite{fo}, see also \eqref{z-x}. 
Under several assumptions on $\vth$, the image of $\gZ(\wg^\vth,t)$  is equal to 
$\gZ_\times(\g,\vth)$. Note that $\gZ_\times(\g,\vth)$ is a `large' Poisson-commutative subalgebra. 
If $\ind\g_{(0)}=\rk\g$, then $\trdeg\gZ_\times(\g,\vth)$ takes the maximal possible value 
$(\dim\g+\rk\g)/2$ \cite{fo}. 

An interesting open problem is whether there are commutative subalgebras in $\U(\q[t])$
with a non-reductive $\q$ that {\it quantise} $\gZ(\wq)$, $\gZ(\wq,[0])$ 
 and in $\U(\g[t])$ that 
{\it quantise}   $\gZ(\wg^\vth,t)$, $\gZ(\wg^\vth,[0])$ with a non-trivial $\vth$.   
%% ... 

\section{Setup and notation}\label{setup}

Let $\cL$ be a Lie algebra, finite- or infinite-dimensional. The symmetric algebra $\gS(\cL)$ is the associated graded of the enveloping 
algebra $\U(\cL)$. Hence it has a Poisson structure $\{\,\,,\,\}$ inherited from the commutator in $\U(\cL)$. 
Here $\{x,y\}=[x,y]$ for $x,y\in\cL$. 
If ${\mathcal A}\subset\U(\cL)$ is a commutative subalgebra, then its graded image 
$A=\gr\!({\mathcal A})\subset\gS(\cL)$ is Poisson-commutative, i.e.,  $\{A,A\}=0$. 
Passing to the commutative object $\gS(\cL)$ allows us to use geometric methods in order to study $\U(\cL)$ and its subalgebras. 
The {\it quantisation problem} %%goes  
starts from a Poisson-commutative 
$A\subset\gS(\cL)$ and asks whether there is a commutative  ${\mathcal A}\subset\U(\cL)$ such that 
$A=\gr\!({\mathcal A})$.  Commutative subalgebras obtained in this way find  applications 
in representation theory of $\cL$, see e.g. \cite{FFR,cris}. 

For elements or subsets $Y_i$  of $\gS(\cL)$ with $1\le i\le k$, let 
$\mathsf{alg}\langle Y_1,\ldots,Y_k\rangle$ be the subalgebra of $\gS(\cL)$ generated 
by $Y_1,\ldots,Y_k$ or by the elements of $Y_1\cup\ldots\cup Y_k$, if $Y_i$ are subsets. 

Recall that $\gt q=\Lie Q$ is a finite-dimensional Lie algebra. 
Let $\q^\xi=\{x\in\q\mid \ad^*(x){\cdot}\xi=0\}$ be the {\it stabiliser\/} in $\q$ of $\xi\in\q^*$. The 
{\it index of\/} $\q$, $\ind\q$, is defined by $\ind\q:=\min_{\xi\in\q^*} \dim \q^\xi$.
The set of {\it regular\/} elements of $\q^*$ is 
\beq       \label{eq:regul-set}
    \q^*_{\sf reg}=\{\eta\in\q^*\mid \dim \q^\eta=\ind\q\}\,.
\eeq
It is a dense open subset of $\q^*$.
Set $\q^*_{\sf sing}=\q^*\setminus \q^*_{\sf reg}$.
We say that $\q$ has the {\sl codim}--$n$ property if $\codim \q^*_{\sf sing}\ge n$. 

The algebra $\gS(\q)^{\q}$ of {\it symmetric invariants of} $\q$ is defined by 
\[
\gS(\q)^{\q}=\{F\in\gS(\q) \mid \{F,\xi\}=0 \ \forall \xi\in\q\}. 
\]
This is the 
{\it Poisson centre of} $(\gS(\q),\{\,\,,\,\})$.

\subsection{Periodic gradings of Lie algebras and related compatible brackets }      
\label{subs:periodic}
Let $\vartheta\in\Aut(\q)$ be a Lie algebra automorphism of finite order $m\ge 1$ and $\zeta=\sqrt[m]1$ 
a primitive root of unity. If $\q_i$ is the $\zeta^i$-eigenspace of $\vartheta$, $i\in \BZ_m$, then the direct sum
$\q=\bigoplus_{i\in \BZ_m}\q_i$ is a {\it periodic grading\/} or $\BZ_m$-{\it grading\/} of 
$\q$. The latter means that $[\q_i,\q_j]\subset \q_{i+j}$ for all $i,j\in \BZ_m$. Here $\q_0=\q^\vartheta$ 
is the fixed-point subalgebra for $\vartheta$ and each $\q_i$ is a $\q_0$-module.
If $\zeta$ is fixed, then we have a bijection between the $\BZ_m$-gradings of $\q$ and the automorphisms 
$\vth\in\Aut(\q)$ such that $\vth^m=\id$. 

We choose $\{0,1,\dots, m-1\}\subset\BZ$ as a fixed set of representatives for $\BZ_m=\BZ/m\BZ$. 
Under this convention, we have
$\q=\q_0\oplus\q_1\oplus\ldots\oplus\q_{m-1}$ and
\beq   \label{eq:Z_m}
[\q_i,\q_j]\subset \begin{cases}  \q_{i+j}, &\text{ if } \ i+j\le m{-}1, \\
 \q_{i+j-m}, &\text{ if } \ i+j\ge m. \end{cases}
\eeq
According to \cite[Sect.\,2]{fo}, there are Lie algebra contractions  $\q_{(0)}=\q_{(0,\vth)}=(\q,[\,\,,\,]_{0})$ and $\q_{(\infty)}=\q_{(\infty,\vth)}=(\q,[\,\,,\,]_\infty)$  such that 
\begin{equation} \label{0andinf}
   [\q_i,\q_j]_0=\begin{cases} [\q_i,\q_j] \  & \text{ if } \ i+j<m \\
       0 & \text{ otherwise}   \end{cases}, \  \ 
    [\q_i,\q_j]_{\infty}=\begin{cases}  [\q_i,\q_j]  \  & \text{ if } \ i+j\ge m  \\
      0 & \text{ otherwise}   \end{cases}.
\end{equation}
In particular, $\q_{(\infty)}$ is nilpotent and the subspace $\q_0$, which is the 
highest grade component of $\q_{(\infty)}$, belongs to the centre of $\q_{(\infty)}$. 
Note that $\q_0$ is a subalgebra of $\q_{(0)}$ isomorphic to $\q^\vth$. 
Furthermore, $\q$, $\q_{(0)}$, and $\q_{(\infty)}$ are isomorphic as $\q^\vth$-modules. 

It is easy to see that $[\,\,,\,]=[\,\,,\,]_0+[\,\,,\,]_\infty$. In other words, the brackets 
$[\,\,,\,]$ and $[\,\,,\,]_0$, as well as the corresponding Poisson brackets $\{\,\,,\,\}$ and $\{\,\,,\,\}_0$, 
are {\it compatible}, i.e.,
any linear combination $a\{\,\,,\,\}+b\{\,\,,\,\}_0$ with $a,b\in\bbk$ is
a Poisson bracket, 
 see \cite[Sect.\,2]{fo} for more details.  Following
 \cite[Sect.\,2.2]{fo}, we associate a Poisson-commutative 
subalgebra  $\gZ=\gZ(\q,\vth)\subset\gS(\q)$ to $\vth$.  
Various properties of these subalgebras are studied in \cite{OY,fo,MZ23,some}. 
If $\vth$ is trivial, then  $[\,\,,\,]=[\,\,,\,]_0$ and $\gZ=\gS(\q)^\q$. 

Set $\cz_0:=\gS(\q_{(0)})^{\q_{(0)}}$ and $\cz_\infty=\gS(\q_{(\infty)})^{\q_{(\infty)}}$.

\subsection{Contractions and invariants}
\label{subs:contr-&-inv} 
We refer to \cite[Ch.\,7,\,\S\,2]{t41} for basic facts on contractions of Lie algebras.
In this article, we consider contractions of the following form. Let $\bbk^{\!^\times}=\bbk\setminus\{0\}$ be 
the multiplicative group of $\bbk$ and 
%let $\vp_s:\q\to \q$, In other words,
$\vp: \bbk^{\!^\times}\to \GL(\q)$, $s\mapsto \vp_s$, a polynomial representation. That is, 
%group 
the matrix entries of $\vp_{s}:\q\to \q$ are polynomials in $s$ w.r.t. some (any) basis of $\q$.
Define a new Lie algebra structure on the vector space $\q$ and associated Lie--Poisson bracket by 
\beq       \label{eq:fi_s}
      [x, y]_{(s)}=\{x,y\}_{(s)}:=\vp_s^{-1}[\vp_s( x), \vp_s( y)], \ x,y \in \q, \ s\in\bbk^{\!^\times}.
\eeq
%%% The corresponding and 
All the algebras $(\q,[\,\,,\,]_{(s)})$  
are isomorphic and  $(\q,[\,\,,\,]_{(1)})$ is the initial Lie algebra $\q$. The induced $\bbk^{\!^\times}\!$-action on the variety of structure constants is not necessarily 
polynomial, i.e., \ $\lim_{s\to 0}[x, y]_{(s)}$ may not exist for all $x,y\in\q$. Whenever all the limits exist, 
we obtain a new linear Poisson bracket, denoted $\{\,\,,\,\}_0$, and thereby a new Lie algebra $\q_{(0)}$, 
which is said to be a {\it contraction\/} of $\q$. If we wish to stress that this construction is determined 
by $\vp$, then we write $\{x, y\}_{(\vp,s)}$ for the bracket in~\eqref{eq:fi_s} and say that $\q_{(0)}=\q_{(0,\vp)}$ is the 
$\vp$-{\it contraction\/} of $\q$ or is the {\it zero limit of $\q$ w.r.t.}~$\vp$.  
A criterion for the existence of $\q_{(0)}$
can be given in terms of Lie brackets of the $\vp$-eigenspaces in $\q$, see~\cite[Sect.\,4]{Y-imrn}. 
We identify all algebras $\q_{(s)}$ and 
$\q_{(0)}$ as vector spaces. The semi-continuity of index implies that $\ind\q_{(0)}\ge \ind\q$.

%There is  
The map $\vp_s$, $s\in\bbk^{\!^\times}$, is naturally extended to an invertible transformation of 
$\gS^j(\q)$, which we also denote by $\vp_s$. The resulting graded map 
$\vp_s\!:\gS(\q)\to\gS(\q)$ is nothing but the comorphism associated with $s\in\bbk^{\!^\times}$ and
the dual representation
$\vp^*\!:\bbk^{\!^\times}\to \GL(\q^*)$.
%%% Since 
Any $F\in\gS^j(\q)$  
can be written as a sum of $\vp_s$-eigenvectors, i.e., 
$F=\sum_{i\ge 0}F_i$, %we have 
where %%%the sum 
$\vp_s(F_i)=s^iF_i\in\gS^j(\q)$. Let $F^\bullet$ denote the non-zero component $F_i$ with maximal $i$.

\begin{prop}[{\cite[Lemma~3.3]{contr}}]     \label{prop:bullet}
If $F\in\gS(\q)^{\q}$ and $\q_{(0)}$ exists, then $F^\bullet\in \gS(\q_{(0)})^{\q_{(0)}}$. 
\end{prop}

Recall that 
%%In the setting of 
in Section~\ref{subs:periodic}, we considered periodic gradings $\q=\bigoplus_{i\in\Z_m}\q_i$. 
In this setting, 
let $\varphi_s$ be given by 
\begin{equation} \label{phi-s}
{\varphi_s}|_{\q_i}=s^i{\cdot}\id.
\end{equation}
Then the contraction $\q\ard\q_{(0,\vth)}$ is defined by $\varphi$.
The contraction $\q\ard\q_{(\infty,\vth)}$ is obtained from the map 
$\bbk^{\!^\times}\to\GL(\q)$ with $s\mapsto s^m \varphi_s^{-1}$.

Suppose that $\q=\g$ is reductive. Then 
$\gS(\g)^{\g}$ has an algebraically independent set of homogeneous generators 
$F_1,\ldots,F_{\rk\g}$. We say that  $F_1,\ldots,F_{\rk\g}$ is a {\it good generating system} in $\gS(\g)^\g$
({\sf g.g.s.}\/ {\it for short}) for $\vp$, if $F_1^\bullet,\dots,F_{\rk\g}^\bullet$ are
algebraically independent. 

Let 
$\varphi$ be associated with $\vth$ as in \eqref{phi-s}. Then $F_1,\ldots,F_{\rk\g}$ is  a {\sf g.g.s.} for $\vth$, if it is 
a {\sf g.g.s.} in $\gS(\g)^{\g}$ for $\varphi$.

\subsection{Invariants of current algebras}
\label{sec-ohne}
Symmetric invariants of certain finite-dimensional Lie algebras are building blocks for  
%%The algebras 
$\gZ(\wq,t^{-1})=\gS(t^{-1}\q[t^{-1}])^{\q[t]}$ and $\gZ(\wq_-,[0]):=\gS(\gt q[t^{-1}])^{\gt q[t]}$.  
 Note that $\gZ(\wq_-,[0])$ and $\gZ(\wq,[0])$ are isomorphic. 
Set 
\[\mW_{-n}=\gt q t^{-n}\oplus\gt q t^{1-n}\oplus\ldots\oplus\gt q t^{-2}\oplus\q t^{-1}\subset \gt q[t,t^{-1}]/\gt q[t].
\] 
Let $\q[t]/(t^n)$ be a {\it truncated current algebra}, also known as 
a {\it  Takiff Lie algebra}   modelled on $\q$. 
%%%%
Then $\gS(\mW_{n})^{\q[t]}\cong \gS(\q[t]/(t^n))^{\q[t]}$. %%, where 
%%%
Since $\q t^N$ with $N\ge n$ acts on $\q[t]/(t^n)$ as zero, the algebra $\gS(\q[t]/(t^n))^{\q[t]}$ is the algebras of symmetric invariants of $\q[t]/(t^n)$. 

Both algebras, $\gZ(\wq,t^{-1})$ and $\gZ(\wq_-,[0])$, 
 have direct limit structures, in particular,  
%% We have 
$\gZ(\wq,t^{-1})=\varinjlim_{n\in\mathbb N} \gS(\mW_{-n})^{\q[t]}$.
Here $\gS(\mW_{-n})^{\q[t]}\subset \gS(\mW_{-n-1})^{\q[t]}$. 
 If we consider $\gZ(\wq_-,[0])$, 
then $\mW_{-n}$ is replaced by $\mW_{-n}\oplus\q$ and 
\[
(\gS(\mW_{-n})\oplus\q)^{\q[t]}\cong \gS(\q[t]/(t^{n+1}))^{\q[t]}.
\] 
Whenever symmetric invariants of each $\q[t]/(t^n)$ with $n\in\mathbb N$ form a polynomial ring, the algebras 
$\gZ(\wq,t^{-1})$ and $\gZ(\wq_-,[0])$ have algebraically independent set of generators. For example, this is the case for 
a reductive $\g$ by \cite{rt}. More results on symmetric invariants of Takiff Lie algebras can be found in 
\cite{kot-T}. 

The FF-centre $\gt z(\wg,t^{-1})$ is the unique quantisation of $\gZ(\wg,t^{-1})$, see \cite{RybU}. By \cite{ff}, it is a polynomial ring.
The same result for $\gt z(\wg,[0])$ is obtained in \cite{ir}. 

\section{Fixed-point subalgebras  of a  loop algebra}%% 
\label{sec-twist}

Any finite   order automorphism $\vartheta\in\Aut(\q)$
can be extended to an automorphism $\Theta$ of the loop algebra $\wq$ by letting 
$\Theta(\xi t^k)=\zeta^k \vth(\xi) t^k$ %%and  
for $\xi\in\q$ and $k\in \Z$. This allows us to introduce the 
%%%
 %%% 
%% 
{\it $\vth$-twisted loop algebra of $\gt q$} as the fixed point subalgebra of $\Theta$,  i.e., 
\[
\wq^\vth=\q [t,t^{-1}]^{\vartheta}:=(\q [t,t^{-1}])^{\Theta}.
\] %%where 
%%% .  
Let $m$ be the order of $\vth$. Then 
%% For instance, 
$\Theta$ acts as  identity on $\q_{m-1}t\subset \q t$. %%In this notation, 
The extension of 
$\vth^{-1}$ acts as $\zeta^{-2}{\cdot}\id$ on $\q_{m-1}t$ and as identity  on $\q_1 t$. Hence it 
is not equal to $\Theta^{-1}$ unless $m=2$. 
If $\vth$ is an outer automorphism of a simple Lie algebra $\gt g$, then  $\g [t,t^{-1}]^{\vartheta}$
%% This 
is related to a twisted Kac--Moody 
algebra, 
see \cite[Chap.~8]{kac} for details.  %%% As in the 
%%Similarly to 
We have 
\begin{gather*}
\q [t,t^{-1}]^{\vartheta} = \ldots\oplus \q_2 t^{-2}\oplus\q_1 t^{-1}\oplus\q_0\oplus\q_{m-1}t \oplus \q_{m-2}t^2\oplus\ldots\,,\\
\q [t,t^{-1}]^{\vartheta^{-1}} = \ldots\oplus \q_{m-2} t^{-2}\oplus\q_{m-1} t^{-1}\oplus\q_0\oplus\q_{1}t \oplus \q_{2}t^2\oplus\ldots\,.
\end{gather*}

Set $\wq_+=\q[t]$.  
We  identify $\wq_+^\vth:=(\wq_+)^\Theta$ with 
$(\q [t,t^{-1}]/t^{-1}\q [t^{-1}])^\Theta$.  This leads to an action of $\q [t^{-1}]^\vth=(\q [t^{-1}])^\Theta$ on 
$\wq_+^\vth$ and hence on $\gS(\wq_+^\vth)$.
We consider the  invariants of $(\q [t^{-1}])^\Theta$ in $\gS(\wq_+^\vth)$, i.e., 
the subalgebra
%in the  
\[
\gZ(\wq^\vth,[0])=\gS(\wq_+^\vth)^{\q [t^{-1}]^\vth}. %% , 
%%%%% \qquad ?
\] 
%% which can 
Furthermore, 
 identifying  $\wq^\vth/\gt q[t^{-1}]^\vth$ with $(t\gt q[t])^\Theta$, 
we regard 
$$
\gZ(\wq^\vth)=\gZ(\wq^\vth,t)=\gS(\wq^\vth/\gt q[t^{-1}]^\vth)^{\gt q[t^{-1}]^\vth}
$$ 
as a subalgerbra of $\gS(t\gt q[t])$.
In the down to earth terms, $\gZ(\wq^\vth)$ consists of the elements  
$Y\in\gS((t\gt q[t])^\Theta)$ such that $\{xt^{k},Y\}\in \gt q[t^{-1}]^\vth\gS(\gt q[t,t^{-1}])$ for each $x t^k\in\q[t^{-1}]^\vth$.
%%%%  and each $k\le 0$. 

In \cite{fo}, we studied the algebra of $\g[t]^\vth$-invariants  $\gZ(\wg^\vth,t^{-1})\subset \gS((t^{-1}\g[t^{-1}])^\Theta)$,
which is naturally isomorphic to $\gZ(\wg^\vth)$. %%% =\gS((\g [t,t^{-1}]/\g [t])^\Theta)^{\g [t]^\vth}$. 
If $\ind\g_{(0)}=\rk\g$, then   $\gZ(\wg^\vth,t^{-1})$ is Poisson-commutative by \cite[Thm.\,8.2{\sf (i)}]{fo}.

The  $\Z$-grading of $\wq^\vth$ is connected with automorphisms of 
direct sums $\q\oplus\ldots\oplus\q$. 
Let $\gt r=\gt q^{\oplus n}$ be a direct sum of $n$ copies of $\q$. %% be a direct sum of $n$ copies of $\h $. 
Let  $\tilde\vartheta\in\Aut(\gt r)$ be the 
composition of $\vartheta$ applied to one copy of $\q$  only  and a cyclic permutation of the
summands. Formally speaking, 
\begin{equation} \label{tilde-vth} 
\tilde\vth\left((y_1,y_2,\ldots,y_n)\right)=(y_n,\vth(y_1),y_2,\ldots, y_{n-1})  \ \ 
\text{ 
for any } \ \  (y_1,y_2,\ldots,y_n)\in\q^{\oplus n}.
\end{equation}
The order of $\tilde\vartheta$ is $N=nm$.   %%As before 
%%Similarly to the case of $\wg_-$, %% untwisted case, 
Note that $\gt r^{\tilde\vth}\cong\q_0$ is embedded in $\q^{\oplus n}$ diagonally. 
Let $\tilde\zeta$ be a primitive $N$-th root of unity 
such that $\tilde\zeta^n=\zeta$. Consider 
the  $\BZ_N$-grading $\gt r=\bigoplus_{i=0}^{N-1}\gt r_i$ associated with $\tilde\vth$ and $\tilde\zeta$
by the rule of Section~\ref{subs:periodic}. 
Then $\gt r_{km+i} \cong \gt q_i$ as a $\gt q_0$-module for all $0\le k<n$ and $0\le i<m$.   

Let $\gt r_{(0)}$ and $\gt r_{(\infty)}$ be the contractions of $\gt r$ related to $\tilde\vth$ 
in the sense of  Section~\ref{subs:periodic}. 
There are isomorphisms  
$\gt r\cong \q [t]^\vartheta/(t^{N}{-}1)$ and $\gt r_{(0)}\cong \q [t]^{\vartheta^{-1}}/(t^{N})$. 
Furthermore $\q [t]^\vartheta/(t^{N+1})\cong \q_0\ltimes\gt r_{(\infty)}$, where 
$\gt q_0$ is a subalgebra of the semi-direct product $\q_0\ltimes\gt r_{(\infty)}$ isomorphic to 
$\gt r^{\tilde\vth}$ and the ideal $\gt r_{(\infty)}$ is a $\gt r^{\tilde\vth}$-module. 

We regard $\q_0^*$ as the annihilator $\Ann(\bigoplus_{i=1}^{m-1} \q_i)\subset\q^*$, 
and similarly  $\gt r_0^*$ as a subspace of $\gt r^*$. 

\begin{lm}\label{0-reg}
Suppose that $\q_0^*\cap\q^*_{\sf reg}\ne\varnothing$. Then 
$\gt r_0^*\cap\gt r^*_{\sf reg}\ne\varnothing$.
\end{lm} 
\begin{proof}
Take any $\xi\in\q_0^*\cap\q^*_{\sf reg}$ and consider 
$\tilde\xi=(\underbrace{\xi,\xi,\ldots,\xi}_{n \ \text{entries}})\in\gt r_0^*\subset\gt r^*$. Here 
$\gt r^{\tilde\xi}=(\q^\xi)^{\oplus n}$ and $\dim\gt r^{\tilde\xi}=n{\cdot}\ind\q=\ind\gt r$.
Therefore $\tilde\xi\in\gt r^*_{\sf reg}$.
\end{proof}

\begin{thm}           \label{twist-q}
 If $\q^*_{\sf reg}\cap\q_0^*\ne\varnothing$, then $\gZ(\wq^\vth,[0])$ is a 
Poisson-commutative subalgebra. 
\end{thm}
\begin{proof}
%% Suppose that $f\in\gS(\wg_-^\vartheta)$ is an invariant of $\h [t]^{\vartheta}$. 
For technical reasons, it is more convenient to deal with the %%change the variable $t\mapsto t^{-1}$ and show that 
isomorphic algebra 
$\gZ(\wq^\vth_-,[0])=\gS((\q [t,t^{-1}]/t\q [t])^\Theta)^{\q [t]^\vth}$ instead of $\gZ(\wq^\vth,[0])$. 
We prove that $\gZ(\wq^\vth_-,[0])$ is Poisson-commutative. 
For $N=nm$, set  
\begin{equation} \label{WN}
    %%  f\in \gS(\mathbb W_N) \ \text { for } \  
     \mathbb W_N:=(\q t^{-N+1}\oplus  \ldots\oplus\q  t^{-1})^\Theta \oplus   \q_0 \subset\q[t^{-1}]^\vth \cong (\q [t,t^{-1}]/t\q [t])^\Theta.
\end{equation}
Then  $\q [t]^{\vartheta}$ acts on $\mathbb W_N$ and any subspace $(\q t^{k})^\Theta$ with $k \ge N$ acts on $\mathbb W_N$ as zero. 
In particular, the quotient %%% $\q [t^{-1}]^{\vartheta}$ acts as 
$\q [t]^\vartheta/(t^{N+1})\cong\q_0\ltimes\gt r_{(\infty)}$ acts on the space  $\mathbb W_N$. %% , where $\gt r=\q^{\oplus n}$,  and 
Furthermore, 
$\mathbb W_N\cong \gt r_{(\infty)}$ as a $(\q_0\ltimes\gt r_{(\infty)})$-module.
Therefore there is an isomorphism of commutative algebras 
$\gS(\mathbb W_N)^{\q [t]^{\vartheta}}\cong \cz_\infty^{\q_0}$, where $\cz_\infty$ is the Poisson centre of 
$\gS(\gt r_{(\infty)})$. 

Suppose $f_1,f_2\in \gS(\q[t^{-1}]^\vartheta)$. Then there is 
$N'=n'm$ such that  $f_1,f_2\in \gS(\mathbb W_{N'})$. Set $n=2n'$. 
Then $\{f_1,f_2\}=0$ if and only if the images $\bar f_1,\bar f_2$ of $f_1, f_2$ in 
$\gS(\q[t^{-1}]^\Theta/(t^{-N}-1))\cong  \gS(\gt r)$  
Poisson-commute. %% The  
%%% $\Z_N$
Suppose $f_1,f_2\in\gZ(\wq^\vth_-,[0])$. Then 
 $\bar f_1,\bar f_2\in \cz_\infty^{\q_0}$ by the construction.  
In view of Lemma~\ref{0-reg}, we have $\gt r_0^*\cap\gt r^*_{\sf reg}\ne\varnothing$. Then 
 by
 \cite[Thm.\,2.3{\sf (i)}]{some}, the algebra $\cz_\infty^{\q_0}=\cz_\infty^{\gt r_0}$ is 
 %%% 
 a Poisson-commutative 
subalgebra of $\gS(\gt r)$. %%% and  %%, which is commutative w.r.t. $\{\,\,,\,\}_{0}$ as well. 
%%% 
Hence $\{\bar f_1,\bar f_2\}=0$. 
\end{proof}

\begin{cl} \label{ohne}
For any finite-dimensional Lie algebra $\q$, we have
$\{\gZ(\wq,[0]),\gZ(\wq,[0])\}=0$. %%% is Poisson-commutative.
\end{cl}
\begin{proof}
If $\vth$ is  trivial, then $\q_0=\q$ and Lemma~\ref{0-reg} and Theorem~\ref{twist-q} apply. %%Thus,  
\end{proof}

\begin{cl} \label{cl-red}
For any finite order automorphism $\vth$ of a reductive Lie algebra $\g$, the algebra $\gZ(\wg^\vth,[0])$ is Poisson-commutative. 
\end{cl}
\begin{proof}
Condition $\g_0^*\cap\g^*_{\sf reg}\ne\varnothing$ is satisfied for any finite order automorphism of 
a {\bf reductive} Lie algebra, see e.g.~\cite[\S 8.8]{kac}.
\end{proof}

Below we describe  several classes of non-reductive Lie algebras satisfying  the assumption of Theorem~\ref{twist-q}. 

\begin{ex}
Consider $\q=\g_{(\infty)}$ associated with a reductive $\g$ and $\vth\in\Aut(\g)$ of order $m$. Then 
$\vth$ is an automorphism of $\q$ and the Abelian subalgebra $\q_0=\q^\vth$ is equal to $\g_0$ as a vector space.  
If $\xi\in\g_0^*$ is regular in $\g^*$, then it is also regular in $\q^*$, see \cite[Lemma\,3.6]{fo}. 
Hence Theorem~\ref{twist-q} applies to $\g_{(\infty)}$ and $\vth$.
\end{ex}

\begin{ex}
In a natural way $\vth$ extends to an automorphism of the semi-direct product $\tilde{\gt g}=\g_0\ltimes\g_{(\infty)}$,
where $\g_0$ is a subalgebra isomorphic to $\g^\vth$ acting on the ideal $\g_{(\infty)}\lhd\tilde{\gt g}$ as on a $\g^\vth$-module.  
Here $\tilde{\gt g}^\vth=\g_0\ltimes\g_0^{\sf ab}$, where $\g_0^{\sf ab}$ is an Abelian ideal. 
It was observed in \cite[Sect.\,4]{some} that any $\xi\in\g^*_{\sf reg}\cap(\g_0^{\sf ab})^*$ is regular in 
$\tilde\g^*$. 
%%%By a similar reason,
Hence Theorem~\ref{twist-q} applies to $\tilde\g$ as well. 
\end{ex}

\begin{thm} \label{g-twist}
 The algebras $\gZ(\wg^\vth,[0])$  and $\gZ(\wg^\vth_-,[0])$
are  polynomial rings in infinitely many variables. %%% if and only if\/ $\gS(\g_{(0)})^{\g_{(0)}}$ is a 
%%%% polynomial ring for each $n\ge 1$
\end{thm}
\begin{proof}
 %%The additional  
It suffices to consider only $\gZ(\wg^\vth_-,[0])$. We may also safely assume that $\g$ is semisimple. 
Similarly to the case of $\gZ(\wq_-,[0])$, see Section~\ref{sec-ohne},
the algebra $\gZ(\wg^\vth_-,[0])=\varinjlim \gS(\mathbb W_{nm})^{\g [t]^{\vartheta}}$ has  
a direct limit structure, cf.~\eqref{WN}.  
Each algebra $\gS(\mathbb W_N)^{\g [t]^{\vartheta}}$ with $N=nm$ 
is isomorphic to $\cz_\infty^{\g_0}\subset\gS(\gt r_{(\infty,\tilde\vth)})$ for $\gt r=\g^{\oplus n}$
and $\tilde\vth$ defined by \eqref{tilde-vth}. 
Recall that $\gt r_0=\gt r^{\tilde\vth}\cong\g_0$, i.e.,  $\cz_\infty^{\g_0}=\cz_\infty^{\gt r_0}$. 
The algebra $\cz_\infty^{\gt r_0}$
is a polynomial ring by \cite[Thm.\,4.6]{some}. 
By a standard argument on graded algebras,  %%% see e.g. the proof of Lemma~\ref{kosiks-lemma}, 
any algebraically independent set of generators of $\gS(\mathbb W_N)^{\g [t]^{\vartheta}}$ extends 
to an algebraically independent set of generators of $\gS(\mathbb W_{N{+}m})^{\g [t]^{\vartheta}}$.
In  the direct limit, one obtains a set of algebraically independent  generators of  $\gZ(\wg_-^\vth,[0])$.
%%% \gS(\wh_-^{\vartheta})^{\h [t]^{\vartheta}}$. 
\end{proof}

Theorem~4.6 of \cite{some} contains an explicit description of generators of $\cz_\infty^{\g_0}\subset\gS(\g^{\oplus n})$. 
This will lead to a description of generators for  $\gZ(\wg^\vth,[0])$ and $\gZ(\wg_-^\vth,[0])$. In order to formulate the relevant results,  some more notation is introduced below. 

For $k\in \mathbb Z$, let $\bar k\in\{0,1,\ldots,m{-}1\}$ be the 
residue of $k$ modulo $m$.  Then 
$$
\g[t,t^{-1}]^{\vartheta}=\bigoplus_{k\in\mathbb Z} \g_{\bar k} t^{-k}.
$$ 
Let $\phi_s\!: (\g[t,t^{-1}]^{\vartheta})^* \to (\g[t,t^{-1}]^{\vartheta})^*$ with $s\in\mK^{^\times}$ be a linear map  multiplying the elements of 
$(\g _{\bar k}t^{-k})^*$ with $s^{k}$. Next we canonically identify $(\g_{\bar k} t^{-k})^*$ with 
$\g_{\bar k}^*$. Take  $\xi\in(\g[t^{-1}]^{\vartheta})^*$ that is given by a finite 
sequence $\xi=(\xi_{0},\xi_{1},\xi_{2},\ldots,\xi_{L})$ with $\xi_{k}\in \g_{\bar k}^*$. 
This means that $\xi(\eta t^{-k})=\xi_{k}(\eta)$ for $\eta\in\g_{\bar k}$ and $0\le k\le L$, while 
$\xi(\eta t^{-k})=0$ for $k>L$. %%A
Set 
\[|\xi|= \sum_{i=0}^L \xi_{i} \in \g^*.
\] 
Note that the construction works with evident changes for a negative integer $L$. %%%  can be negative as 
If $F\in\gS(\g)$, then for each $s\in\mK^{^\times}$ and each $\xi$ as above, 
$$F(|\phi_s(\xi)|)=\sum_{k\ge 0}  s^kF_{[k,L]}(\xi),
$$
 where
$F_{[k,L]}\in\gS(\bigoplus_{i=0}^L \g_{\bar i} t^{-i})$ and 
 the sum is actually finite. 
If $L\ge k$, then $F_{[k,L]}$ does not depend on $L$. Therefore 
set $F_{[k]}=F_{[k,k]}$. 
We say that elements $F_{[k]}$ are {\it $t$-polarisations} of $F$.  If $F$ is an eigenvector of $\vth$,
i.e., $\vth(F)=\zeta^u F$ with $0\le u<m$, then $F_{[k]}=0$, if $\bar k\ne u$. 
If we take  $\xi=(\xi_L,\ldots,\xi_{-1},\xi_0)\in(\g[t]^{\vartheta})^*$ with $L\le 0$, then 
%% $L$ is negative, 
$F(|\phi_s(\xi)|)$ decomposes as 
a sum over $k\le 0$  and  here $F_{[k]}\in\gS(\g[t]^\vth)$. 

Let $\varphi\!:\bbk^{\!^\times}\to\GL(\g)$ be the map defined by~\eqref{phi-s} with $\q_i$ replaced by $\g_i$. 
 Let $X=\xi_1\ldots\xi_d$ be a monomial such that each 
$\xi_j\in\g$ is an eigenvector of $\vth$. Suppose that $\varphi_s(X)=s^\ell X$. %%%, where $0\le \ell<m$.
Let $\bar\xi_i\in\g[t^{-1}]$ be a copy of $\xi_i$ obtained under the natural identification of vector spaces 
$\g\cong\g_{m-1}t^{1-m}\oplus\ldots\oplus\g_1 t^{-1}\oplus\g_0$.  
Applying the above procedure, we obtain the following description of 
 $t$-polarisations of $X$. %%%  are described by the  formula.
For each $j\ge 0$, %% we have  then 
\begin{equation}\label{pol-t} 
X_{[\ell+jm]} \ \text{ is a sum of the monomials } \
\bar\xi_1 t^{\alpha_1 m}\ldots \bar\xi_d t^{\alpha_d m} \ \text{ with } \  \sum_{i=1}^d \alpha_i=-j \ \text{ and }  \ \alpha_i\le 0.
\end{equation}

From now on, assume that $\g$ is semisimple. 
Let $r=\rk\g$ be the rank of $\g$.
Since $\vth$ is an automorphism of $\g$, it acts on $\gS(\g)^{\g}$ and there is a set of 
of homogeneous  generators $\{F_i\mid 1\le i\le r\}\subset\gS(\g)^{\g}$ that consists of $\vartheta$-eigenvectors. 
Suppose that $\vth(F_i)=\zeta^{\ell_i} F_i$, where $0\le\ell_i<m$.  Set $d_i=\deg F_i$. 

\begin{thm}        \label{twist-inf}
Let $\{F_1,\ldots,F_r\}\subset\gS(\g)^{\g}$ be a set of homogeneous  generators that are $\vartheta$-eigenvectors. 
Then 
$\gZ(\wg_-^\vartheta,[0])$ is freely generated by
basic symmetric invariants of $\g_0$ and 
 the nonzero 
$t$-polarisations $(F_i)_ {[k]}$ with $k >0$. %%%  related to the subspaces $((\h  t^{-j})^{\vartheta})^*$. 
The same statement holds for $\gZ(\wg^\vartheta,[0])$, if we take $k<0$. 
\end{thm}
\begin{proof}
We use the direct limit structure $\gZ(\wg_-^\vth,[0])=\varinjlim \gS(\mathbb W_{nm})^{\g [t]^{\vartheta}}$,
where $\mathbb W_{nm}$ with $nm=N$ is the same as in \eqref{WN}. 
If $n=1$, then  $\mathbb W_m\cong\gt g_{(\infty)}$ and $\gS(\mathbb W_m)^{\g [t]^{\vartheta}}\cong \cz_{\infty}^{\gt g_0}$,
where $\cz_{\infty}=\gS(\g_{(\infty)})^{\g_{(\infty)}}$. We identify $\mathbb W_m$ ang $\g_{(\infty)}$. 
By \cite[Thm.\,4.6]{some},  $\cz_{\infty}^{\gt g_0}$ is generated by $\gS(\g_0)^{\g_0}$ and the 
highest components $F_{i}^{\bullet}$, related to the 
contraction 
$\gt g\ard\g_{(\infty)}$, of the invariants $F_i$
 such that $\vth(F_i)=\zeta^{\ell_i} F_i$ with $0<\ell_i<m$.
%%% 
In our current notation, $F_{i}^{\bullet}=(F_i)_{[\ell_i]}$, %%. %%%%, 
see~\cite[Eq.\,(4${\cdot}$7)]{some}.   %%\eqref{in}.   

Suppose now that $n\ge 2$ and $\gt r=\g^{\oplus n}$. Set $N=nm$.
Let $\gt r_{(\infty)}=\gt r_{(\infty,\tilde\vth)}$ be the contraction of $\gt r$ associated with $\tilde\vth$, where $\tilde\vth$ is defined by 
 \eqref{tilde-vth}.  Recall from the proof of Theorem~\ref{twist-q} that 
 the commutative algebra $\gS(\mathbb W_{N})^{\g[t]^\vth}$ is isomorphic to $\cz_\infty^{\g_0}\subset\gS(\gt r_{(\infty)})$.
 
Let $\tilde\zeta$ be a primitive $N$-th root of unity 
such that $\tilde\zeta^n=\zeta$. Set
 $\omega=\tilde\zeta^m$. 
We identify $\g$ with the first direct summand of $\gt r$. This leads to an embedding 
$\gS(\g)\subset \gS(\gt r)$.  
Then $\gS(\gt r)^{\gt r}$ is generated 
by 
\begin{equation}              \label{ir}
     H_{rj+i}=\frac{1}{n}\left(F_i+\omega^j\tilde\zeta^{-\ell_i}\tilde\vth(F_i)+\omega^{2j}\tilde\zeta^{-2 \ell_i}\tilde\vth^2(F_i)+\ldots + \omega^{(n-1)j}\tilde\zeta^{(1-n) \ell_i}\tilde\vth^{n-1}(F_i)\right), 
     \end{equation}
with $1\le i\le r$ and $0\le j<n$. Furthermore,  
\[
\tilde\vth(H_{rj+i})=\omega^{(n-1)j} \tilde\zeta^{\ell_i} H_{rj+i}, \ \ \text{ where } \ \ 
\omega^{(n-1)j} \tilde\zeta^{\ell_i} =\tilde\zeta^{\ell_i-mj}=\tilde\zeta^{m(n-j)+\ell_i}.
\] 
Note that $\tilde\zeta^{\ell_i-mj}=1$ if and only if $\ell_i=j=0$.  
By \cite[Thm.\,4.6]{some},  $\cz_{\infty}^{\g_0}\subset\gS(\gt r_{(\infty)})$  is generated by $\gS(\g_0)^{\g_0}$ and 
the 
highest components
$H_{rj+i}^{\bullet}$, related to the 
contraction $\gt r\ard\gt r_{(\infty)}$, of the invariants $H_{rj+i}$ 
such that $\tilde\vth(H_{rj+i})\ne H_{rj+i}$. %%% with $0<r_i<m$.
Under the identification of vector spaces $\gt r\cong \mathbb W_N$,  
each $H_{rj+i}^{\bullet}$ is a $t$-polarisation of $F_i$.  It follows from \cite[Eq.\,(4${\cdot}$7)]{some}, that whenever  $\tilde\vth(H_{rj+i})\ne H_{rj+i}$, the highest component 
$H_{rj+i}^{\bullet}$ is equal to $n^{-d_i}(F_i)_{[\ell_i]}$, if $j=0$, and to 
$n^{-d_i}(F_i)_{[m(n-j)+\ell_i]}$, if $j>0$.

Let $\{\bh_1,\ldots,\bh_{\rk\g_0}\}\subset\gS(\g_0)^{\g_0}$ be a generating set. 
The equality $\gZ(\wg_-^\vartheta,[0])=\bigcup\limits_{n\ge 1} \gS(\mathbb W_{nm})^{\g[t]^{\vth}}$ and induction on 
$n$ show 
%%% Arguing inductively, we see
 that  
$\gZ(\wg_-^\vartheta,[0])$ is freely generated by the union
\[
\{\bh_i\mid 1\le i\le{\rk\g_0}\}\cup \{(F_i)_{[mj+\ell_i]} \mid \vth(F_i)\ne F_i, \ j\ge 0\}      \cup
\{ (F_i)_{[mj]} \mid \vth(F_i)=F_i, \ j\ge 1\}. 
\]
We finish the argument by considering other $t$-polarisations $(F_i)_{[k]}$. 

Let $\hat\vth\in\Aut(\g[t,t^{-1}])$ be an 
extension of $\vth$ obtained by letting $\hat\vth(t)=t$,
$\hat\vth(t^{-1})=t^{-1}$, i.e., $\hat\vth(\xi t^k)=\vth(\xi)t^k$ for $\xi\in\g$ and $k\in \Z$.  
Then $\hat\vth\ne\Theta$, if $m>1$.
%%  different 
Note that $\Theta$ and $\hat\vth$ commute in $\Aut(\g[t,t^{-1}])$.  
By the construction, 
$\hat\vth((F_i)_{[k]})=\zeta^{\ell_i}(F_i)_{[k]}$ and at the same time $\hat\vth((F_i)_{[k]})=\zeta^k (F_i)_{[k]}$.  
Hence $(F_i)_{[k]}=0$, if $k-\ell_i\not\in m\Z$. If $\vth(F_i)=F_i$, then $(F_i)_{[0]}\in\gS(\g_0)^{\g_0}$ and it is not needed in 
a generating set.
\end{proof}

\begin{rmk}\label{comp}
We compare results on $\gZ(\wg^\vartheta,[0])$ and $\gZ(\wg^\vth,t^{-1})$. \\[.2ex]
{\sf (i)} The former algebra is always Poisson-commutative. For the latter, the statement is proven under the condition that 
$\ind\g_{(0)}=\rk\g$ \cite{fo}. Most probably, this condition holds for all reductive Lie algebras, see \cite[Conj.\,3.1]{MZ23}.  \\[.3ex]
{\sf (ii)} 
If $\g_{(0)}$ satisfies even stronger conditions, then 
$\gZ(\wg^\vth,t^{-1})$ is a polynomial ring by \cite[Thm.\,8.2\,\&\,Prop.\,8.3]{fo}. 
However, 
if $\gS(\g_{(0)})^{\g_{(0)}}$ is not a polynomial ring, then $\gZ(\wg^\vth,t^{-1})$ is not either, see \cite[Thm.\,8.2(ii)]{fo}. There are examples such that $\gS(\g_{(0)})^{\g_{(0)}}$
does not have an algebraically independent set of generators \cite{Y-imrn}.  In contrast, Theorems~\ref{g-twist} and \ref{twist-inf} apply to all reductive Lie algebras asserting that $\gZ(\wg^\vartheta,[0])$ is a polynomial ring. 
\end{rmk}

\section{Finite-dimensional quotients of loop algebras}  \label{Q}

In \cite[Sect.\,7]{fo} and \cite{mult}, we have studied the images of  $\gZ(\wq)$ in $\gS(\q[t,t^{-1}]/(p))$ 
for polynomials $p$ such that $p(0)\ne 0$. In order to study similar quotients for $\g[t,t^{-1}]^\vth$, instead 
of $\g[t,t^{-1}]$, one has to assume that  %% the polynomial
$p$ is an eigenvector of $\Theta$ for the action given by $\Theta(t)=\zeta t$. The simplest and most interesting example is provided by $p=t^m-1$.   By the construction, there is
%%% Recall that there is 
an isomorphism of Lie algebras $\g[t,t^{-1}]^\vth/(t^m-1)\cong\g$.  Let 
$$
\psi\!: \gS(\g[t,t^{-1}]^\vth)\to  \gS(\g[t,t^{-1}]^\vth)/(t^m-1)=\gS(\g[t,t^{-1}]^\vth/(t^m-1))\cong\gS(\g)
$$
 be the quotient map.  
%%%% .... images of these  ........ 

In \cite{fo}, we have used maps $\varphi_s\!:\g\to\g$ such that $s\in\bbk^{\!^\times}$ and 
${\varphi_s}|_{\g_k}=s^k{\cdot}\id$ for each $k$, cf.~\eqref{phi-s}. Then in 
\cite[Sect.\,4]{fo}, each generator $F_i\in\gS(\g)^{\g}$ was decomposed as $F_i=\sum_{j\ge 0} F_{i,j}$, where 
$\varphi_s(F_{i,j})=s^j F_{i,j}$. In \cite{fo}, 
we introduced a Poisson-commutative subalgebra  $\gZ=\gZ(\g,\vth)\subset\gS(\g)$ associated with $\vth$.
It is shown there that 
$\gZ$ contains a subalgebra 
\begin{equation}\label{z-x}
\gZ_{\times}=\gZ_{\times}(\g,\vth)=\mathsf{alg}\langle F_{i,j} \mid 1\le i \le r \ \text{ and } \ 0\le j \rangle
\end{equation}
generated by $ F_{i,j} $. 
If $\ind\g_{(0)}=\rk\g$ and $\g_0$ is not Abelian, then $\gZ=\mathsf{alg}\langle \gZ_\times,\cz_0\rangle$;
if $\ind\g_{(0)}=\rk\g$ and $\g_0$ is Abelian, then $\gZ=\mathsf{alg}\langle \gZ_\times,\cz_0,\g_0\rangle$ 
\cite{fo,some}.

\begin{thm} \label{image1}
We have $\psi(\gZ(\wg^\vth_-,[0]))=\mathsf{alg}\langle \gZ_\times,\gS(\g_0)^{\g_0}\rangle$.
\end{thm}
\begin{proof}
By Theorem~\ref{twist-inf}, $\gZ(\wg^\vth_-,[0])$ is generated by $\gS(\g_0)^{\g_0}$ and 
$(F_i)_ {[k]}$ with $k>0$ such that $k-\ell_i\in m\Z$. Clearly $\psi(\gS(\g_0)^{\g_0})=\gS(\g_0)^{\g_0}$. 
Next we consider the images of $(F_i)_ {[k]}$  under $\psi$.

If $0<k<m$ and $(F_i)_{[k]}\ne 0$, then $k=\ell_i$ and $\psi((F_i)_{[k]})=F_{i,\ell_i}$ in the above notation. 
%%% of \cite[Sect.\,4]{fo}. 

Recall that $d_i=\deg F_i$.
%%% Suppose that $F_{i,\ell_i+jm}$ with
For $j\ge 1$, %% is nonzero. Then 
the image $\psi((F_i)_{[\ell_i+jm]})$ 
can be calculated using \eqref{pol-t} and it 
is equal to
\begin{equation} \label{im1}
F_{i,\ell_i+jm}+d_i F_{i,\ell_i+(j-1)m}+ \ldots + \binom{d_i+k-1}{k-1} F_{i,\ell_i+(j-k)m} +
\ldots +  \binom{d_i+j-1}{j-1} F_{i,\ell_i}.  
\end{equation}

For any $i$, we have $F_i =\sum_{j\ge 0} F_{i,\ell_i+jm} $, where the sum is  
finite and $F_{i,0}$ belongs to  $\gS(\g_0)$. 
By definition,  
$$
\mathsf{alg}\langle \gZ_\times,\gS(\g_0)^{\g_0}\rangle=\mathsf{alg}\langle F_{i,\ell_i+jm},\bh_u
\mid 1\le i\le r, \ 0\le j, \ 1\le u\le \rk\g_0\rangle. 
$$
%%% The above formula  %%%% ......  the following observation.
%%%
%%% .......... 
Arguing by induction on $j$ and using~\eqref{im1}, we prove that $\psi(\gZ(\wg^\vth_-,[0]))$ and $\mathsf{alg}\langle \gZ_\times,\gS(\g_0)^{\g_0}\rangle$ have the same generators. 
\end{proof}

If we consider the quotient of $\gZ(\wg^\vth,[0])$ by $t^m-1$, then $\gZ_\times$ should be replaced 
by $\gZ_\times(\g,\vth^{-1})\subset\gZ(\g,\vth^{-1})$. Recall that  a {\sf g.g.s.} for $\vth$ is defined in 
Section~\ref{subs:contr-&-inv}. 

\begin{thm} \label{image2}
Suppose $\ind\g_{(0)}=\rk\g$, there is a {\sf g.g.s.} $F_1,\ldots,F_r$ in $\gS(\g)^{\g}$ for $\vartheta$, and 
$\g_{(0)}$ has the {\sl codim}--$2$ 
property. 
Then $\psi(\gZ(\wg^\vth))=\gZ_\times=\gZ$.
%%%
\end{thm}
\begin{proof}
Thanks to \cite[Lemma\,4.4]{fo}, we may assume that each $F_i$ in our {\sf g.g.s.} is an eigenvector of $\vth$.  Then 
a set of generators of $\gZ(\wg^\vth)$  is described by   
\cite[Prop.\,8.3]{fo}. That proposition is formulated for $\gZ(\wg^\vth,t^{-1})$, but under the 
assumption that $\vth$ multiplies $t$ with $\zeta^{-1}$.  
%%making obvious changes.  
%%%% Each of 
Keeping this in mind, we deduce that 
the algebra $\gZ(\wg^\vth)$ is freely generated by $t$-polarisations 
$(F_i)_{[k]}$ with $k<0$. Furthermore,  the generators related to $F_i$ are of the form 
$(F_i)_{[-b_i+jm]}$, where $j\le 0$ and 
$$
(F_i)_{[-b_i]}\in \gS(\g_{m-1}t\oplus\g_{m-2}t^2\oplus\ldots \oplus \g_1 t^{m-1}\oplus \g_0 t^m)\cong\gS(\g_{(0)}) 
$$
 is 
equal to the highest component $F_i^\bullet$ of $F_i$ with respect to the contraction $\g\ard\g_{(0)}$.   

By \cite[Sect.\,4]{fo}, 
$$
F_i=F_{i,\ell_i}+F_{i,\ell_i+m}+F_{i,\ell_i+2m}+\ldots+F_{i,d_i^\bullet},
$$
where $F_{i,d_i^\bullet}=F_i^\bullet$ and $d_i^\bullet\in\ell_i+m\Z$ is the $\varphi_s$-degree of $F_i$. Then 
$b_i=md_i-d_i^\bullet$.  We have $\psi((F_i)_{[-b_i]})=(F_i)_{[-b_i]}=F_{i,d_i^\bullet}$. %%Furthermore, 
The other images  $\psi((F_i)_{[-b_i+jm]})$ are calculated using \eqref{pol-t}. 
If $F_{i,d_i^\bullet+jm}$ with $j<0$ is nonzero, then $d_i^\bullet+jm\ge \ell_i$ and 
$\psi((F_i)_{[-b_i+jm]})$ is equal to
\[
F_{i,d_i^\bullet+jm}+d_i F_{i,d_i^\bullet+(j-1)m}+ \ldots + \binom{d_i+k-1}{k-1} F_{i,d_i+(j-k)m} +
\ldots +  \binom{d_i^\bullet+j-1}{j-1} F_{i,d_i^\bullet},
\]
cf.~\eqref{im1}. % and \eqref{pol-t}.  
 If $F_{i,d_i^\bullet+jm}=0$, then 
$\psi((F_i)_{[-b_i+jm]})$  is still  a linear combination of $F_{i,\ell_i+km}$ with $k\ge 0$.

We see that  $\psi(\gZ(\wg^\vth))$ is contained in $\gZ_\times$. Furthermore, 
by induction on $j$, each generator $F_{i,d_i^\bullet+jm}\in\gZ_\times$ is contained in $\psi(\gZ(\wg^\vth))$. 
%% each image 
Hence 
$\psi(\gZ(\wg^\vth))=\gZ_\times$. Under our assumptions on $\g_{(0)}$, we have 
$\gZ_\times=\gZ$ by \cite[Cor.\,4.7]{fo}   and \cite[Cor.\,4.8]{some}.
\end{proof}

Theorems~\ref{image1} and \ref{image2} show 
%%% We see 
a certain similarity between $\psi(\gZ(\wg^\vth))$ and $\psi(\gZ(\wg^\vth_-,[0]))$. 
There is a small difference, 
the image of $\gZ(\wg^\vth)$ does not contain $\gS(\g_0)^{\g_0}$, 
while $\gS(\g_0)^{\g_0}$ lies in $\psi(\gZ(\wg^\vth_-,[0]))$.
Since $\{\gZ(\wg^\vth),\g_0\}$ is zero,  this can be remedied by 
extending $\gZ(\wg^\vth)$. 

A more significant difference is %% that 
%%% In that  
%% $\ind\g_{(0)}=\rk\g$ and 
related to 
$\cz_0$. 
%
%%%
On the one hand, there are automorphisms $\vth$ such that $\cz_0\not\subset \psi(\gZ(\wg^\vth_-,[0]))$,
see Example~\ref{E6} below. On the other hand, 
%%% Recall that 
$\cz_0$ is contained in $\gZ(\wg^\vth)$ under a natural identification 
of $\g_{m-1}t\oplus\ldots \oplus\g_1 t^{m-1}\oplus \g_0 t^m \subset\g[t]^\vth$ with $\g$ 
and hence 
$\cz_0\subset\psi(\gZ(\wg))$. 

One can deduce from the proof of \cite[Prop.\,8.3]{fo} that each polarisation $(F_i)_{[k]}$ with $k\le 0$ belongs to 
$\gZ(\wg^\vth)$. Repeating the proof of Theorem~\ref{image2} and using a description of $\gZ$ obtained in 
\cite{fo,some}, we conclude that 
$\psi(\mathsf{alg}\langle\gZ(\wg^\vth),\gS(\g_0)^{\g_0}\rangle)$ contains $\mathsf{alg}\langle\gZ,\cz_0,\gS(\g_0)^{\g_0}\rangle$
for any $\vth$. However, 
%%%  
%%% 
there are cases, where a description of $\psi(\gZ(\wg^\vth))$ is unknown.
Therefore it can happen that the image  of  $\mathsf{alg}\langle\gZ(\wg^\vth),\gS(\g_0)^{\g_0}\rangle$ is even larger.

\begin{ex}\label{E6}
Let $\g$ be simple of type ${\sf E}_6$ and $\vth$ an inner involution with 
$\g_0=\gt{so}_{10}\oplus\gt{so}_2$. Then there is no  {\sf g.g.s.}  for $\vth$ in $\gS(\g)^{\g}$ \cite[Remark\,4.3]{p07}. 
But $\ind\g_{(0)}=6=\rk\g$  \cite[Sect.\,2]{p07}.  
The algebra $\gS(\g)^{\g}$ has $6$ homogeneous generators $F_i$ of degrees $2,5,6,8,9,12$.  Assume that %%Choose the numbering such that  
%% Assume that 
$\deg F_1=2$. %%%% <\deg F_{i+1}$ for each $i<6$. 
The algebra $\gS(\g_0)^{\g_0}$ has $6$ generators $\bh_u$ of degrees $1,2,4,5,6,8$. 
By Theorem~\ref{image1} and \eqref{z-x},  
 $\psi(\gZ(\wg^\vth_-,[0]))$ is generated by $\bh_1,\bh_2,\bh_3,\bh_4,\bh_5,\bh_6$  and 
$F_{i,j}$, where  $1\le i\le 6$ and $0<j$. 

The algebra $\cz_0$ contains $\gS(\g_1)^{\g_0}$, which is a polynomial ring with two generators 
$\bff_1,\bff_2$ such that $\deg\bff_1=2$ and $\deg\bff_2=4$. It is not difficult to see 
that  $F_1=F_{1,0}+F_{1,2}$ and 
that 
$\bff_1=F_{1,2}$ up to a nonzero scalar. 
Since $\bff_2\in\gS(\g_1)$ is not proportional to $\bff_1^2$, 
 $\bff_2\not\in\gZ_\times$ and also 
%% We can say even more, 
$\bff_2\not\in\mathsf{alg}\langle \gZ_\times,\gS(\g_0)^{\g_0}\rangle=\psi(\gZ(\wg^\vth_-,[0]))$
by degree considerations. 
Thus 
$\cz_0\not\subset \psi(\gZ(\wg^\vth_-,[0]))$. 

In this case, $\cz_0$ is not a polynomial ring \cite[Sect.\,6.3]{Y-imrn}. No generating set for $\cz_0$  is known. 
No generating set is known for $\gS(\g_1 t + \g_0 t^2+\g_1t^3+\g_0 t^4)^{\g[t^{-1}]^\vth}\subset \gZ(\wg^\vth)$ 
or any 
$$
\gS\left(\bigoplus_{j=1}^{2n} \g_{\bar j}t^j\right)^{\g[t^{-1}]^\vth} \subset \gZ(\wg^\vth)
$$
either. Hence a description of 
 $\gZ(\wg^\vth)$ and $\psi(\gZ(\wg^\vth))$ remains a mistery.
\end{ex}

\begin{rmk} \label{Rcomp}
By Theorems~\ref{image1} and \ref{image2}, there is a surjective homomorphism ${\psi\!:\g[t,t^{-1}]\to\g}$ such that 
 $\psi(\gZ(\wg^\vth))$ and $\psi(\gZ(\wg^\vth,[0]))$ contain $\gZ_\times(\g,\vth)$.
This means that these images are quite large.  
If $\ind\g_{(0)}=\rk\g$, then $\trdeg\gZ_\times(\g,\vth)=(\dim\g+\rk\g)/2$, which is  the maximal possible value, see  \cite[Sect.\,3]{fo}.  The condition $\ind\g_{(0)}=\rk\g$ holds for many automorphisms, in particular, for 
all involutions, see \cite[Sect.\,2]{p07} and  \cite[Sect.\,3\,\&\,4]{MZ23}. 
 %% ........ {\it hm}. 
 
%%Consider a similar image of $\gZ(\wg)$ or $\gZ(\wg,[0])$. 
Let $\gamma\!:\g[t,t^{-1}]\to \g$ be a   surjective homomorphism and set 
$\gt a=\ker\gamma$. Suppose that $\g$ is simple (non-Abelian).   Then $\gt a\cap\g=\{0\}$,  since $[\g,\g[t,t^{-1}]]=\g[t,t^{-1}]$. 
Hence the restriction of $\gamma$ to $\g$ is an isomorphism. By the construction, 
$\gZ(\wg)$ and $\gZ(\wg,[0])$ consist of $\g$-invariant elements. Thus, their images under $\gamma$ 
are contained in $\gS(\g)^{\g}$. Here $\trdeg \gS(\g)^{\g}=\rk\g$. 
Thus, there is no surjective homomorphism $\gamma\!:\g[t,t^{-1}]\to \g$ such that the transcendence degree of 
$\gamma(\gZ(\wg))$ or $\gamma(\gZ(\wg,[0]))$ is equal to $(\dim\g+\rk\g)/2$. 
This shows that $\gZ(\wg)$ and  $\gZ(\wg,[0])$ are quite different in nature from 
$\gZ(\wg^\vth)$ and $\gZ(\wg^\vth,[0])$. 
\end{rmk}

\end{document}